\documentclass[11pt,fullpage]{article}

\usepackage{graphicx}
\usepackage{latexsym}
\usepackage{color}
\usepackage{multirow}
\usepackage{amsmath}
\usepackage{graphicx}
\usepackage{amssymb}
\usepackage{amsthm}
\usepackage{amsfonts}

\newtheorem{theorem}{Theorem}
\newtheorem{lemma}{Lemma}

\newtheorem*{assumption}{Assumption}
\newtheorem{corollary}{Corollary}

\long\def\ignore#1{}

\newcommand{\be}{\begin{equation}}
\newcommand{\ee}{\end{equation}}
\newcommand{\beqn}{\begin{eqnarray}}
\newcommand{\eeqn}{\end{eqnarray}}

\newcommand{\bfm}[1]{\mbox{\boldmath $#1$}}
\newcommand{\bbeta}{\bfm{\beta}}
\newcommand{\btheta}{\bfm{\theta}}
\newcommand{\bzeta}{\bfm{\zeta}}
\newcommand{\bxi}{\bfm{\xi}}

\newcommand{\bY}{{\bf Y}}
\newcommand{\bx}{{\bf x}}

\newcommand{\mE}{\mathbb{E}}
\newcommand{\cu}{{\cal U}}
\newcommand{\cl}{{\cal L}}
\newcommand{\cb}{{\cal B}}

\newcommand{\M}{\mathfrak{M}}

\begin{document}

\title{\bf Model selection and minimax estimation in generalized linear models}
\author{{\bf Felix Abramovich} \\
Department of Statistics and Operations Research \\
Tel Aviv University
\and
{\bf Vadim Grinshtein} \\
Department of Mathematics and Computer Science \\
The Open University of Israel}

\date{}

\maketitle

\begin{abstract}
We consider model selection in generalized linear models (GLM) for
high-dimensional data and propose a wide class of
model selection criteria based on
penalized maximum likelihood with a complexity penalty on the model size.
We derive a general nonasymptotic upper bound for the Kullback-Leibler risk
of the resulting estimators and establish the corresponding minimax lower bounds for sparse GLM.
For the properly chosen (nonlinear) penalty,
the resulting penalized maximum likelihood estimator is shown to be
asymptotically minimax and adaptive to the unknown sparsity. We discuss
also possible extensions of the proposed approach to model selection
in GLM under additional structural constraints and aggregation.
\end{abstract}

\bigskip

\section{Introduction} \label{sec:intr}
Regression analysis of high-dimensional data, where the number of
potential explanatory variables (predictors) $p$ might be even large
relative to the sample size $n$ faces a severe ``curse of
dimensionality'' problem. Reducing the dimensionality of the model
by selecting a sparse subset of ``significant'' predictors becomes
therefore crucial. The interest to model selection in regression
goes back to seventies (e.g., seminal papers of Akaike, 1973;
Mallows, 1973 and Schwarz, 1978), where the considered ``classical''
setup assumed $p \ll n$. Its renaissance started in 2000s with the
new challenges brought to the door of statistics by exploring data,
where $p$ is of the order of $n$ or even larger. Analysing the ``$p$
larger than $n$'' setup required novel approaches and techniques, and
led to novel model selection procedures. The corresponding theory
(risk bounds, oracle inequalities, minimax rates, variable selection consistency, etc.) for model
selection in Gaussian linear regression has been intensively
developed in the literature in the last decade. See Foster \& George (1994),
Birg\'e \& Massart (2001, 2007), Chen \& Chen (2008),
Bickel {\em et al.} (2009), Abramovich \&
Grinshtein (2010), Raskutti {\em et al.} (2011), Rigollet \&
Tsybakov (2011) among many others.
A review on model selection in Gaussian regression for ``$p$ larger than $n$''
setup can be found in Verzelen (2012).

Generalized linear models (GLM) is a generalization of Gaussian linear
regression, where the distribution of response is not necessarily
normal but belongs to the exponential family of distributions.
Important examples include binomial and Poisson data arising in
a variety of statistical applications. Foundations of a general theory for GLM
have been developed in McCullogh \& Nelder (1989).

Although most of the
proposed model selection criteria for Gaussian regression have been
extended and are nowadays widely used in GLM (e.g., AIC of Akaike, 1973
and BIC of Schwarz, 1978), not much is known on their theoretical
properties in the general GLM setup. There are some results on variable
selection consistency of several model selection criteria
(e.g., Fan \& Song, 2010; Chen \& Chen,
2012), but a rigorous theory of model selection for estimation and prediction in
GLM remains essentially {\em terra incognita}. We can mention
van de Geer (2008) that investigated the Lasso estimator
in GLM and
Rigollet (2012) that considered 
aggregation problem for GLM. The presented paper intends to
contribute to fill this gap.

We introduce a wide class of model selection criteria for GLM based on the penalized maximum
likelihood estimation with a complexity penalty on the model size. In particular, it includes
AIC, BIC and some other well-known criteria.
In a way, this approach can be viewed as an extension of that of
Birg\'e \& Massart (2001, 2007) for Gaussian regression.
We derive a general nonasymptotic upper bound for the Kullback-Leibler risk
of the resulting estimator.
Furthermore, for the properly chosen penalty
we establish its asymptotic minimaxity and 
adaptiveness to the unknown sparsity. Possible extensions to model selection
under additional structural constraints and aggregation are also discussed.

The paper is organized as follows.
The penalized maximum likelihood model selection procedure for GLM is
introduced in Section \ref{sec:ms}. Its main theoretical properties are
presented in Section \ref{sec:main}. In particular, we derive the upper
bound for its Kullback-Leibler risk
and corresponding minimax lower bounds, and establish its asymptotic minimaxity
over various sparse settings. We illustrate the obtained general results on the example of logistic regression. Extensions to model selection
under structural constraints and aggregation are discussed in
Section \ref{sec:extensions}. All the proofs are given in the Appendix.

\section{Model selection procedure for GLM} \label{sec:ms}
\subsection{Notation and preliminaries} \label{subsec:setup}
Consider a GLM setup with a response variable $Y$ and a set of $p$
predictors $x_1,...,x_p$. We observe a series of independent
observations $(\bx_i,Y_i),\;i=1,\ldots, n$, where the design points $\bx_i \in
\mathbb{R}^p$ are deterministic, and the distribution $f_{\theta_i}(y)$ of $Y_i$
belongs to a (one-parameter) natural exponential family with a natural
parameter $\theta_i$ and a scaling parameter $a$:
\be
\label{eq:model} f_{\theta_i}(y)=\exp\left\{\frac{y \theta_i -
b(\theta_i)}{a}+c(y,a)\right\}
\ee
The function $b(\cdot)$ is assumed to be twice-differentiable.
In this case
$\mE(Y_i)=b'(\theta_i)$ and $Var(Y_i)=ab''(\theta_i)$ (see McCullagh
\& Nelder, 1989).
To complete GLM we assume the canonical link
$\theta_i=\bbeta^{t}\bx_i$ or, equivalently, in the matrix form,
$\btheta=X\bbeta$, where $X_{n \times p}$ is the design matrix and
$\bbeta \in \mathbb{R}^p$ is a vector of the unknown regression
coefficients.

In what follows we assume
the following assumption on the parameter space $\Theta$
and the second derivative $b''(\cdot)$:
\begin{assumption}[\bf A] \label{as:A}
\noindent
\begin{enumerate}
\item Assume that $\theta_i \in \Theta$, where the parameter
space $\Theta \subseteq \mathbb{R}$ is a closed (finite or infinite) interval.
\item Assume that there exist constants $0 < \cl \leq \cu < \infty$ such that the
function $b''(\cdot)$ satisfies the following conditions:
\begin{enumerate}
\item $\sup_{t \in \mathbb{R}} b''(t) \leq \cu$
\item $\inf_{t \in \Theta} b''(t) \geq \cl$
\end{enumerate}
\end{enumerate}
\end{assumption}
Similar assumptions were imposed in van de Geer (2008) and Rigollet (2012). Conditions on $b''(\cdot)$ in
Assumption (A) are intended to exclude two degenerate
cases, where the variance $Var(Y)$ is infinitely large or small. They
also ensure strong convexity of $b(\cdot)$ over
$\Theta$. For Gaussian distribution, $b''(\theta)=1$ and, therefore,
$\cl=\cu=1$ for any $\Theta$. For the binomial distribution,
$b''(\theta)=\frac{e^\theta}{(1+e^\theta)^2},\;\cu=\frac{1}{4}$,
while the condition $(b)$ is equivalent to the boundedness of $\theta:
\Theta=\{\theta: |\theta| \leq C_0\}$, where $\cl=\frac{e^{C_0}}{(1+e^{C_0})^2}$.

Let $f_{\btheta}$ and $f_{\bzeta}$ be two possible joint distributions of
the data from the exponential family with $n$-dimensional vectors of natural parameters
$\btheta$ and $\bzeta$ correspondingly.
A Kullback-Leibler divergence $KL(\btheta, \bzeta)$ between
$f_{\btheta}$ and $f_{\bzeta}$ is then
\be
\begin{split}
KL(\btheta, \bzeta)&=\mE_{\btheta}\left\{ \ln
\left(\frac{f_{\btheta}(\bY)}{f_{\bzeta}(\bY)}\right)\right\}=\frac{1}{a}~\mE_{\btheta}
\left\{\sum_{i=1}^n Y_i(\theta_i-\zeta_i)-b(\theta_i)+b(\zeta_i)\right\} \\
&=
\frac{1}{a}\sum_{i=1}^n \left\{b'(\theta_i)(\theta_i-\zeta_i)-b(\theta_i)+b(\zeta_i)\right\}=
\frac{1}{a}\left(b'(\btheta)^t(\btheta-\bzeta)-(b(\btheta)-b(\bzeta))^t
{\bf 1}\right),
\nonumber
\end{split}
\ee
where $b(\btheta)=(b(\theta_1),\cdots,b(\theta_n))$ and $b(\bzeta)=(b(\zeta_1),\ldots,b(\zeta_n))$.

For a given estimator $\widehat{\btheta}$ of the unknown $\btheta$ consider the Kullback-Leibler loss
$KL(\btheta,\widehat{\btheta})$ -- the Kullback-Leibler divergence between the
true distribution $f_{\btheta}$ of the data and 
its empirical distribution $f_{\widehat{\btheta}}$ generated by $\widehat{\btheta}$. The goodness of $\widehat{\btheta}$ is measured
by the corresponding Kullback-Leibler risk:
\be \label{eq:mE}
\mE KL(\btheta,\widehat{\btheta})=
\frac{1}{a}\left(b'(\btheta)^t(\btheta-\mE(\widehat{\btheta}))-(b(\btheta)-\mE b(\widehat{\btheta}))^t,
{\bf 1}\right)
\ee
where the expectation is taken w.r.t. the true distribution $f_{\btheta}$.
In particular, for the Gaussian case,
where $b(\theta)=\theta^2/2$ and $a=\sigma^2$, $\mE KL(\btheta,\widehat{\btheta})$ is the mean squared error
$E||\widehat{\btheta}-\btheta||^2$ divided by the constant $2\sigma^2$. The binomial
distribution will be considered in more details in Section \ref{subsec:logist} below.

\subsection{Penalized maximum likelihood model selection} \label{subsec:penlik}
Consider a GLM (\ref{eq:model}) with a vector of natural parameters
$\btheta$ and the canonical link $\btheta=X\bbeta$. 

For a given model $M \subset \{1,\ldots,p\}$ consider
the corresponding maximum likelihood estimator (MLE) $\widehat{\bbeta}_M$ of
$\bbeta$:
\be \label{eq:mle}
\widehat{\bbeta}_M=\arg \max_{\widetilde{\bbeta} \in \cb_M} \ell(\widetilde{\bbeta})=
\arg \max_{\widetilde{\bbeta} \in \cb_M}
\left\{\sum_{i=1}^n (Y_i (X\widetilde{\beta})_i-b((X\widetilde{\beta})_i)\right\}
=
\arg \max_{\widetilde{\bbeta} \in \cb_M} \left\{\bY^t X\widetilde{\bbeta}-b(X\widetilde{\bbeta})^t {\bf 1}\right\},
\ee
where $\cb_M =\{\bbeta \in \mathbb{R}^p: \beta_j=0\;{\rm if}\; j \not\in M\; {\rm and}\; \bbeta^t \bx_i \in
\Theta\;{\rm for\;all}\;i=1,\ldots,n \}$.
Note that generally $\cb_M$ depends on a given design matrix $X$.
Except Gaussian regression, the MLE $\widehat{\bbeta}_M$ in (\ref{eq:mle})
is not available in the closed
form but can be obtained numerically by the iteratively reweighted
least squares algorithm (see McCullogh \& Nelder, 1989, Section 2.5).

The MLE for $\btheta$ is $\widehat{\btheta}_M=X\widehat{\bbeta}_M$, and the
ideally selected model (oracle choice)
is then the one that minimizes $\mE
KL(\btheta,\widehat{\btheta}_M)=\frac{1}{a}
\left(b'(\btheta)^t(\btheta-\mE(\widehat{\btheta}_M))-(b(\btheta)- \mE
b(\widehat{\btheta}_M))^t {\bf 1}\right) $ or, equivalently, $
-b'(\btheta)^t\mE(\widehat{\btheta}_M)+\mE b(\widehat{\btheta}_M)^t {\bf 1}
$ over $M$ from the set of all $2^p$ possible models $\M$. An oracle chosen model depends evidently on the
unknown $\btheta$ and can only be used as a benchmark for
any available model selection procedure.

Consider instead an empirical analog $KL([b']^{-1}(\bY),\widehat{\btheta}_M)$
of $\mE KL(\btheta,\widehat{\btheta}_M)$, where
the true $\mE \bY = b'(\btheta)$ is replaced by $\bY$. A naive approach of
minimizing $KL([b']^{-1}(\bY),\widehat{\btheta}_M)$ yields maximizing
$\bY^t \widehat{\btheta}_M-b(\widehat{\btheta}_M)^t{\bf 1}$ (or, equivalently,
maximizing $\ell(\widehat{\bbeta}_M)$) over $M \in \M$ and
obviously leads to the saturated model.

A common remedy to avoid such a trivial unsatisfactory choice is to add a
complexity penalty
$Pen(|M|)$ on the model size $|M|$ and consider the corresponding {\em penalized}
maximum likelihood model selection criterion of the form
\be \label{eq:penlik}
\widehat{M}=\arg \max_{M \in \M} \left\{\ell(\widehat{\bbeta}_M) - Pen(|M|)\right\}=
\arg \min_{M \in \M} \left\{\frac{1}{a}
\left(b(X\widehat{\bbeta}_M)^t{\bf 1}-\bY^t X\widehat{\bbeta}_M\right)+Pen(|M|) \right\},
\ee
where the MLE $\widehat{\bbeta}_M$ for a given model $M$ are given in (\ref{eq:mle}).
The properties of the resulting model selection procedure depends crucially
on the choice of the complexity penalty.
The commonly used criteria for model selection in GLM are $AIC=-2 \ell(\widehat{\bbeta}_M)+ 2|M|$ of Akaike (1973),
$BIC=-2\ell(\widehat{\bbeta}_M)+ |M| \ln n$ of Schwarz (1973) and its
extended version $EBIC=-2\ell(\widehat{\bbeta}_M)+ |M| \ln n+2\gamma |M| \ln p,
\; 0 \leq \gamma \leq 1$ of Chen \& Chen (2012)
correspond to $Pen(|M|)=|M|$, $Pen(|M|)=\frac{|M|}{2} \ln n $ and
$Pen(|M|)=\frac{|M|}{2} \ln n+\gamma |M| \ln p$ in (\ref{eq:penlik})
respectively. A similar extension of RIC criterion
$RIC=-2\ell(\widehat{\bbeta}_M)+ 2|M| \ln p$ of Foster \& George (1994)
yields $Pen(|M|)=|M| \ln p$. Note that all the above penalties increase
{\em linearly} with a model size $|M|$.

\section{Main results} \label{sec:main}
In this section we investigate theoretical properties of the penalized maximum
likelihood model selection procedure proposed in Section \ref{subsec:penlik}
and discuss the optimal choice for the complexity penalty $Pen(|M|)$ in
(\ref{eq:penlik}). We start from deriving a (nonasymptotic) upper bound for the expected
Kullback-Leibler risk of the resulting estimator for a given
$Pen(|M|)$ and then establish its asymptotic minimaxity for a properly chosen
penalty. To illustrate the general results we consider the example of logistic regression.

\subsection{General upper bound for the Kullback-Leibler risk}
\label{subsec:generalupper} Consider a GLM (\ref{eq:model}) with the
canonical link $\btheta=X\bbeta$ and the natural parameters
$\theta_i \in \Theta$ satisfying Assumption (A). Let $r=rank(X)$.
The number of possible predictors $p$ might be larger than the
sample size $n$. We assume that any $r$ columns of $X$ are linearly
independent and consider only models of sizes at most $r$ in
(\ref{eq:penlik}) since otherwise, for any $\bbeta \in \cb_M$, where
$|M|>r$, there necessarily exists another $\bbeta^*$ with at most
$r$ nonzero entries such that $X\bbeta=X\bbeta^*$.

We now present an upper bound for the Kullback-Leibler risk
of the proposed maximum penalized likelihood estimator
valid for a wide class of penalties. Moreover, it does not
require the GLM assumption on the canonical link $\btheta=X\bbeta$ and
can still be applied when a link function is misspecified.

\begin{theorem} \label{th:generalupper}
Consider a GLM (\ref{eq:model}),
where $\theta_i \in \Theta,\;i=1,\ldots,n$ and let Assumption (A) hold.

Let $L_k,\;k=1,\ldots,r$ be a sequence of positive weights such that
\be \label{eq:weights}
\sum_{k=1}^{r-1} \binom{p}{k} e^{-kL_k}+e^{-rL_r} \leq S
\ee
for some absolute constant $S$ not depending on $r$, $p$ and $n$.

Assume that the complexity penalty $Pen(\cdot)$
in (\ref{eq:penlik}) is such that
\be \label{eq:penalty}
Pen(k) \geq 2~\frac{\cu}{\cl}~ k (A+2\sqrt{2 L_k}+4 L_k),\;\;\;k=1,\ldots,r
\ee
for some $A>1$.

Let $\widehat{M}$ be a model selected in (\ref{eq:penlik}) with $Pen(\cdot)$
satisfying (\ref{eq:penalty}) and $\widehat{\bbeta}_{\widehat M}$
be the corresponding MLE estimator (\ref{eq:mle}) of $\bbeta$.
Then,
\be \label{eq:generalupper}
\mE KL(\btheta, X\widehat{\bbeta}_{\widehat M})
\leq \frac{4}{3}~ \inf_{M \in \M} \left\{\inf_{\widetilde{\bbeta} \in \cb_M} KL(\btheta,X\widetilde{\bbeta})+Pen(|M|)\right\}+
\frac{16}{3}~\frac{\cu}{\cl}~ \frac{2A-1}{A-1}~ S
\ee
\end{theorem}

The term $\inf_{\widetilde{\bbeta} \in \cb_M} KL(\btheta,X\widetilde{\bbeta})$ in
(\ref{eq:generalupper}) can be interpreted as a Kullback-Leibler divergence
between a true distribution $f_{\btheta}$ of the data and the family of
distributions $\{f_{X\widetilde{\bbeta}},\;\widetilde{\bbeta} \in \cb_M\}$
generated by the span of a subset of columns of $X$
corresponding to the model $M$. The binomial coefficients $\binom{p}{k}$
appearing in the
condition (\ref{eq:weights}) for $1\leq k<r$ are the numbers of all possible models of size $k$.
The case $k=r$
is treated slightly differently in (\ref{eq:weights}).
For $p=r$, there is evidently a single saturated model.
For $p>r$, although there are
$\binom{p}{r}$ various models of size $r$, all of them are nevertheless
undistinguishable in terms of $X\bbeta_M$ and can be still associated with a
single (saturated) model.

For Gaussian regression, $\mE KL(X\bbeta, X\widehat{\bbeta}_{\widehat M})=
\frac{1}{2\sigma^2}\mE||X\bbeta-X\widehat{\bbeta}_{\widehat M}||^2$,
$\min_{\widetilde{\bbeta} \in \cb_M} KL(X\bbeta,X\widetilde{\bbeta})=
\frac{1}{2\sigma^2}||X\bbeta-
X\bbeta_M||^2$, where $X\bbeta_M$ is the projection of $X\bbeta$ on the
span of columns of $M$, $\cl=\cu=1$ and the upper bound (\ref{eq:generalupper}) is similar (up to
somewhat different constants) to those of
Birg\'e \& Massart (2001, 2007).
Thus, Theorem \ref{th:generalupper} essentially  extends their results for GLM.

Consider two possible choices of weights $L_k$ and the corresponding penalties.

\noindent
{\em 1. Constant weights}. The simplest choice of the weights $L_k$'s is to take them equal, i.e. $L_k=L$ for all $k=1,\ldots,r$.
The condition (\ref{eq:weights}) implies then
$$
\sum_{k=1}^{r-1} \binom{p}{k} e^{-kL}+e^{-rL} \leq \sum_{k=1}^p \binom{p}{k} e^{-kL}=
(1+e^{-L})^p-1
$$
The above sum is bounded by an absolute constant for $L=\ln p$.
It can be easily verified that for $L=\ln p$ and $p \geq 3$,
there exists $A>1$ such that $A+2\sqrt{2L}+4L \leq 8 L$.
Thus,
$2~ \frac{\cu}{\cl}~ k (A+2\sqrt{2L}+4L)  \leq
16~ \frac{\cu}{\cl}~ k  \ln p$
that implies the RIC-type {\em linear} penalty
\be \label{eq:conpen}
Pen(k)=C~\frac{\cu}{\cl}~ k \ln p,\;\;\;k=1,\ldots,r
\ee
in Theorem \ref{th:generalupper} with $C \geq 16$.

Note that the AIC criterion corresponding
to $Pen(k)=k$ (see Section  \ref{subsec:penlik}) does not satisfy
(\ref{eq:penalty}).

\bigskip

\noindent
{\em 2. Variable weights.} Using the inequality $\binom{p}{k} \leq \left(\frac{pe}{k}\right)^k$
(see, e.g., Lemma A1 of Abramovich {\em et al.}, 2010), one has
\be \label{eq:varweights}
\sum_{k=1}^{r-1} \binom{p}{k} e^{-kL_k} + e^{-r L_r} \leq \sum_{k=1}^{r-1}
\left(\frac{pe}{k}\right)^k e^{-kL_k} + e^{-rL_r}= \sum_{k=1}^{r-1}
e^{-k(L_k-\ln(pe/k)}+e^{-rL_r}
\ee
that suggests the choice of $L_k \sim c \ln\left(\frac{pe}{k}\right),\;k=1,\ldots,r-1$ and $L_r=c$  for some $c>1$, and leads
to the {\em nonlinear} penalty of the form
$Pen(k) \sim C~\frac{\cu}{\cl}~  k \ln\left(\frac{pe}{k}\right)$ for $k=1,\ldots,r-1$ and
$Pen(r) \sim C~\frac{\cu}{\cl}~ r$ for
some constant $C$.

More precisely, for any $C>16$ there exist constants
$\widetilde{C}, A > 1$ such that $C \geq 16 A \tilde{C}$. Define $L_k=\widetilde{C}\ln\left(\frac{pe}{k}\right),\;k=1,\ldots,r-1$ and
$L_r=\widetilde{C}$. From (\ref{eq:varweights}) one can easily verify
the condition (\ref{eq:weights}) for such weights $L_k$.
Furthermore, for $1 \leq k \leq r-1$ we have
\be \nonumber
\begin{split}
2~\frac{\cu}{\cl} k (A+2\sqrt{2L_k}+4L_k)  < &~ 2 A~\frac{\cu}{\cl} k \left((1+\sqrt{2L_k})^2+2L_k\right)~ < ~ 2 A~ \frac{\cu}{\cl} k \left((1+\sqrt{2})^2 L_k+2L_k\right) \\
\leq & ~16 A~ \frac{\cu}{\cl} k L_k  
~\leq  ~ C~ \frac{\cu}{\cl} k \ln\left(\frac{pe}{k}\right)
\end{split}
\ee
and similarly, for $k=r$,
$$
2~\frac{\cu}{\cl}~ r (A+2\sqrt{2L_r}+4L_r)
\leq C~\frac{\cu}{\cl}~ r
$$
The corresponding (nonlinear) penalty in (\ref{eq:penalty}) is therefore
\be \label{eq:varpen}
Pen(k)=C~\frac{\cu}{\cl}~ k \ln \left(\frac{pe}{k}\right),\;k=1,\ldots,r-1 \;\;
{\rm and} \;\; Pen(r)= C~\frac{\cu}{\cl}~r,
\ee
where $C>16$. For Gaussian regression such $k\ln\frac{p}{k}$-type penalties were considered in
Birg\'e \& Massart (2001, 2007), Bunea {\em et al.} (2007), Abramovich \& Grinshtein (2010) and Rigollet
\& Tsybakov (2011).

The choice of $C > 16$ in (\ref{eq:conpen}) and 
(\ref{eq:varpen}) was mostly motivated by simplicity of calculus and it may possibly be reduced by more accurate analysis.

\subsection{Risk bounds for sparse models} \label{subsec:sparserisk}
Theorem \ref{th:generalupper} established a general upper bound for the
Kullback-Leibler risk without any assumption on the size of a true underlying model. 
Analysing large data sets
it is commonly assumed that only a subset of predictors has a real impact on
the response. Such extra {\em sparsity} assumption becomes especially
crucial for ``$p$ larger than $n$''
setups. We now show that for sparse models the upper bound (\ref{eq:generalupper}) can be improved.

For a given $1 \leq p_0 \leq r$, consider a set of models of size
at most $p_0$. Obviously, $|M| \leq p_0$
iff the $l_0$ (quasi)-norm of regression
coefficients $||\bbeta||_0 \leq p_0$, where
$||\bbeta||_0$ is the number of nonzero entries.
Define $\cb(p_0)=\{\bbeta \in \mathbb{R}^p: \bbeta^t \bx_i \in \Theta\;
{\rm for \; all}\;i=1,\ldots,n,\;{\rm and}\;||\bbeta||_0 \leq p_0\}$.

Consider a GLM with the canonical link $\btheta=X\bbeta$ under
Assumption (A), where $\bbeta \in \cb(p_0)$.
We refine the general upper bound (\ref{eq:generalupper})
for a penalized maximum likelihood estimator (\ref{eq:penlik})
with a RIC-type linear penalty (\ref{eq:conpen}) and
a nonlinear $k\ln\frac{p}{k}$-type penalty (\ref{eq:varpen}) considered in
Section \ref{subsec:generalupper} for sparse models with $\bbeta \in \cb(p_0)$.

Apply the general upper bound (\ref{eq:generalupper}) with  $A$
corresponding to the chosen constant $C$ in (\ref{eq:conpen}) and (\ref{eq:varpen}) (see
Section \ref{subsec:generalupper}), and the true
$\widetilde{\btheta}=X\widetilde{\bbeta},\;\widetilde{\bbeta} \in \cb(p_0)$ in the RHS.
For both penalties, we then have
\be
\label{eq:upper0} \sup_{\bbeta \in \cb(p_0)} \mE
KL(X\bbeta,X\widehat{\bbeta}_{\widehat M}) \leq 
\frac{4}{3}~ Pen(p_0)+ \frac{16}{3}~\frac{\cu}{\cl}~ \frac{2A-1}{A-1}~ S
\leq C_1 Pen(p_0)
\ee
for some constant $C_1>4/3$ not depending on $p_0$, $p$ and $n$.

Thus, for the RIC-type penalty (\ref{eq:conpen}), (\ref{eq:upper0})
yields $\sup_{\bbeta \in \cb(p_0)} \mE KL(X\bbeta,X\widehat{\bbeta}_{\widehat M})
=O(p_0\ln p)$, while for the nonlinear $k\ln\frac{p}{k}$-type penalty (\ref{eq:varpen})
the Kullback-Leibler risk is of a smaller order
$O\left(p_0 \ln(\frac{pe}{p_0})\right)$.
Moreover, the latter can be improved further for dense models, where
$p_0 \sim r$. Indeed, for a saturated model of size $r$ in the RHS of (\ref{eq:generalupper}), the penalty (\ref{eq:varpen}) yields
\be \label{eq:satmodel}
\sup_{\bbeta \in \cb(p_0)} \mE KL(X\bbeta,X\widehat{\bbeta}_{\widehat M})
\leq \sup_{\bbeta \in \cb(r)} \mE KL(X\bbeta,X\widehat{\bbeta}_{\widehat M})
\leq C_1 Pen(r) = O(r)
\ee
and the final upper bound for an estimator with the penalty (\ref{eq:varpen}) 
is, therefore,
\be \label{eq:upper}
C_1~ \frac{\cu}{\cl}~ \min\left(p_0 \ln\frac{pe}{p_0},r\right)
\ee
with $C_1>4/3$.

To assess the accuracy of the upper bound (\ref{eq:upper})
we establish the corresponding lower bound for the
minimax Kullback-Leibler risk over $\cb(p_0)$.

We introduce first some additional notation.
For any given $k=1,\ldots,r$, let $\phi_{min}[k]$ and $\phi_{max}[k]$
be the $k$-sparse minimal and maximal eigenvalues of the design defined as
$$
\phi_{min}[k]=\min_{\bbeta: 1 \leq ||\bbeta||_0 \leq k}
\frac{||X\bbeta||^2}{||\bbeta||^2},
$$
$$
\phi_{max}[k]=\max_{\bbeta: 1 \leq ||\bbeta||_0 \leq k}
\frac{||X\bbeta||^2}{||\bbeta||^2}
$$
In other words, $\phi_{min}[k]$ and $\phi_{max}[k]$ are respectively
the minimal and maximal eigenvalues of all $k \times k$ submatrices
of the matrix $X^tX$ generated by any $k$ columns of $X$.
Let $\tau[k]=\phi_{min}[k]/\phi_{max}[k],\;k=1,\ldots,r$.

\begin{theorem} \label{th:lower}
Consider a GLM with the canonical link $\btheta=X\bbeta$
under Assumptions (A).

Let $1 \leq p_0 \leq r$ and assume that
$\widetilde{\cb}(p_0) \subseteq \cb(p_0)$, where
the subsets $\widetilde{\cb}(p_0)$ are defined in the proof.
Then, there exists a constant $C_2>0$ such that
\be
\label{eq:lower}
\inf_{\widehat{\btheta}} \sup_{\bbeta \in \cb(p_0)}
\mE KL(X\bbeta,\widehat{\theta}) \geq
\left\{
\begin{array}{ll}
C_2~\frac{\cl}{\cu}~
\tau[2p_0]\; p_0 \ln\left(\frac{pe}{p_0}\right),&\; 1 \leq p_0 \leq r/2  \\
C_2~\frac{\cl}{\cu}~  \tau[p_0]\; r, &\; r/2 \leq p_0 \leq r
\end{array}
\right.
\end{equation}
where the infimum is taken over all estimators $\hat{\btheta}$ of $\btheta$.
\end{theorem}

\subsection{Asymptotic adaptive minimaxity} \label{subsec:minmax}
We consider now the asymptotic properties of the proposed penalized MLE estimator as
the sample size $n$ increases. The number of predictors $p=p_n$ may increase
with
$n$ as well, where we allow $p>n$ or even $p \gg n$. In such asymptotic setup
there is essentially a {\em sequence} of design matrices $X_{n,p_n}$, where
$r_n=rank(X_{n,p_n})$. For simplicity of notation, in what
follows we omit the index $n$ and denote $X_{n,p_n}$ by $X_p$ to highlight the
dependence on $p$, and let $r$ tend to infinity. Similarly,
we define the corresponding sequences of regression coefficients
$\bbeta_p$ and sets $\cb_p(p_0)$.
The considered asymptotic GLM setup can now be viewed as a sequence of GLM
models of the form $Y_i \sim f_{\theta_i}(y),\;i=1,\ldots,n$, where
$f_{\theta_i}(y)$ are given in (\ref{eq:model}), $\theta_i \in \Theta$,
$\btheta=X_p \bbeta_p$ and $rank(X_p)=r \rightarrow \infty$.

As before, we assume that any $r$ columns of $X_p$ are linearly
independent and, therefore, $\tau_p[r]>0$. We distinguish between
two possible cases: {\em weakly collinear} design, where the
sequence $\tau_p[r]$ is bounded away from zero by some constant
$c>0$, and {\em multicollinear} design, where $\tau_p[r] \rightarrow
0$. Intuitively, it is clear that weak collinearity of the design
cannot hold when $p$ is ``too large'' relative to $r$. Indeed,
Abramovich \& Grinshtein (2010, Remark 1) showed that for weakly
collinear design, necessarily $p=O(r)$ and, thus, $p=O(n)$.

For weakly collinear design the following corollary is an immediate
consequence of (\ref{eq:upper}) and Theorem \ref{th:lower}:
\begin{corollary} \label{cor:minimax}
Consider a GLM  with the
canonical link and weakly collinear design.
Then, as $r$ increases, under Assumption (A) and other assumptions of Theorem \ref{th:lower} the
following statements hold~:
\begin{enumerate}
\item The asymptotic minimax  ullback-Leibler
risk from the true model
over $\cb_p(p_0)$ is of the order $\min\left(p_0 \ln\left(\frac{pe}{p_0}\right),r\right)$ or
essentially $p_0\ln\left(\frac{pe}{p_0}\right)$ (since $p=O(r)$ -- see comments above),
that is, there exist two constants $0< C_1 \leq C_2 < \infty$ depending
possibly on the ratio $\frac{\cu}{\cl}$ such that for
all sufficiently large $r$,
$$
C_1~ p_0 \ln\left(\frac{pe}{p_0}\right)  \leq
\inf_{\widehat{\btheta}}\sup_{\bbeta_p \in \cb_p(p_0)}
\mE KL(X_p \bbeta_p,\widehat{\btheta})
\leq C_2~ p_0 \ln\left(\frac{pe}{p_0}\right)
$$
for all $1 \leq p_0 \leq r$.

\item Consider penalized maximum likelihood model selection rule (\ref{eq:penlik})
with the complexity penalty $Pen(k)= C~\frac{\cu}{\cl}~ k\ln\left(\frac{pe}{k}\right),\;
k=1,\ldots,r-1$ and $Pen(r)=C~\frac{\cu}{\cl}~ r,$ where $C>16$.
Then, the resulting penalized MLE estimator $X_p\widehat{\bbeta}_{p\widehat M}$
attains the minimax convergence rates (in terms
of $\mE KL(X_p\bbeta_p,X_p\widehat{\bbeta}_{p\widehat M})$)
simultaneously over all $\cb_p(p_0),\;1 \leq p_0 \leq r$.
\end{enumerate}
\end{corollary}

Corollary \ref{cor:minimax} is a generalization of the corresponding results
of Abramovich \& Grinshtein (2010) for Gaussian regression. It also shows
that model selection criteria with RIC-type (linear) penalties (\ref{eq:conpen})
of the form
$Pen(k)=C k \ln p$ are of the minimax order for sparse models with $p_0 \ll p$
but only suboptimal otherwise.

We would like to finish this section with several important remarks:

\medskip

\noindent {\bf Remark 1.} Under Assumption (A), $KL(\btheta,\bzeta)
\asymp ||\btheta-\bzeta||^2$ (see Lemma \ref{lem:equiv} in the Appendix) and
Corollary \ref{cor:minimax} implies then that
$X_p\widehat{\bbeta}_{p\widehat{M}}$ is also a minimax-rate estimator
for natural parameters $\btheta=X_p\bbeta_p$ in terms of the quadratic risk
simultaneously over all $\cb_p(p_0),\;p_0=1,\ldots,r$.
Furthermore, since
$||X_p\widehat{\bbeta}_{p\widehat{M}}-X_p\bbeta_p||^2 \asymp
||\widehat{\bbeta}_{p\widehat{M}}- \bbeta_p||^2$ for weakly
collinear design, the same is true for
$\widehat{\bbeta}_{p\widehat{M}}$ as an estimator of the regression
coefficients $\bbeta_p \in \cb_p(p_0)$.

\medskip
\noindent {\bf Remark 2.} As we have mentioned, multicollinear
design necessarily appears when $p \gg n$. For such type of design,
$\tau_p[r]$ tends to zero, and there is a gap in the rates in the
upper and lower bounds (\ref{eq:upper}) and (\ref{eq:lower}).
Somewhat surprisingly, multicollinearity, being a ``curse'' for
consistency of variable selection or estimation of regression coefficients
$\bbeta$, may be a ``blessing'' for
estimating $\btheta=X\bbeta$.
For Gaussian regression Abramovich \&
Grinshtein (2010) showed that strong correlations between predictors can be
exploited to reduce the size of a model (thus, to decrease the
variance) without paying much extra price in the bias and, therefore, to
improve the upper bound (\ref{eq:upper}). The analysis
of multicollinear case is however much more delicate and technical
even for the linear regression (see Abramovich \& Grinshtein, 2010), and we do not
discuss its extension for GLM in this paper.

\medskip
\noindent {\bf Remark 3.}
Like any model selection criteria based on complexity penalties, minimization in (\ref{eq:penlik}) is a nonconvex optimization problem that generally
requires search over all possible models. To make computations practically feasible
for high-dimensional data, common approaches are either various greedy algorithms
(e.g., forward selection) that approximate the global solution of (\ref{eq:penlik}) by
a stepwise sequence of local ones, or convex relaxation methods, where the
original nonconvex problem is replaced by a related convex program. The most well-known and 
well-studied method is the celebrated Lasso (Tibshirani, 1996). For {\em linear}
complexity penalties of the form $Pen(|M|)=C |M|$  it replaces the original $l_0$-norm $|M|=||\widehat{\bbeta}_M||_0$ in (\ref{eq:penlik}) by the $l_1$-norm $||\widehat{\bbeta}_M||_1$. Theoretical properties of Lasso for
Gaussian regression have been intensively studied in the literature during the last
decade (see, e.g., Bickel, Ritov \& Tsybakov, 2009). Van de Geer (2008) investigated
Lasso in the GLM setup but with random design. In particular, she showed that under assumptions similar to Assumption (A) and some additional restrictions on the design, Lasso with a properly chosen tuning parameter $C$
behaves similar to the RIC-type estimator and its Kullback-Leibler risk achieves the sub-optimal rate $O(p_0 \ln p)$.

\subsection{Example: logistic regression} \label{subsec:logist}
We now illustrate the obtained general results on logistic regression. 

Consider the Bernoulli distribution $Bin(1,p)$. A simple calculus shows that it belongs to the natural exponential family with the natural parameter $\theta=\ln \frac{p}{1-p},\;b(\theta)=\ln(1+e^{\theta})$ and $a=1$. Thus, $b''(\theta)=\frac{e^{\theta}}{(1+e^{\theta})^2} \leq 1/4$ and, as we have already mentioned in Section \ref{subsec:setup}, the condition (a) of Assumption (A) is satisfied with $\cu=1/4$ for any $\Theta$, while the condition (b)
is equivalent to the boundedness of $\theta$:
$\Theta=\{\theta: |\theta| \leq C_0\}$, where $\cl=\frac{e^{C_0}}{(1+e^{C_0})^2}$. 
In terms of the original parameter of the binomial distribution $p=\frac{e^\theta}{1+e^\theta}$ it means that $p$ is bounded away from zero and one:
$\delta \leq p \leq 1-\delta$ for some $0 < \delta < 1/2$ and $\cl=\delta(1-\delta)$ . 
The same restriction on $p$ appears in van de Geer (2008).

Consider now a logistic regression, where a binary data $Y_i \sim Bin(1,p_i),\; 
\bx_i \in \mathbb{R}^p$ are deterministic and 
$\ln \frac{p_i}{1-p_i}=\bbeta^t \bx_i,\;i=1,\ldots,n$. Following (\ref{eq:mle}), for a given model $M$, the MLE of
$\bbeta$ is
\be \label{eq:binlik}
\widehat{\bbeta}_M=\arg \max_{\widetilde{\bbeta} \in \cb_M}
\sum_{i=1}^n\left \{\bx_i^t \widetilde{\bbeta}_M Y_i -  \ln\left(1+\exp(\bx_i^t \widetilde{\bbeta}_M)\right)\right\},
\ee
where $\cb_M$ was defined in (\ref{eq:mle}).
The MLE for the resulting probabilities $p_{Mi}$'s are $\widehat{p}_{Mi}=
\frac{\exp(\widehat{\bbeta}_M \bx_i)}{1+\exp(\widehat{\bbeta}_M \bx_i)},\;i=1,\ldots,n$.

The model $\widehat{M}$ is selected w.r.t. the penalized maximum likelihood model selection criterion 
(\ref{eq:penlik}):
\be \label{eq:binpen}
\widehat{M}=\arg \min_{M \in \M} 
\left\{\sum_{i=1}^n \left( \ln\left(1+\exp(\bx_i^t \widehat{\bbeta}_M)\right)-
\bx_i^t \widehat{\bbeta}_M Y_i \right) + Pen(|M|) \right\}
\ee

A straightforward calculus shows that the Kullback-Leibler divergence $KL({\bf p}_1,{\bf p}_2)$ between two sample distributions from $Bin(1,p_{1i})$ and
$Bin(1,p_{2i}),\;i=1,\ldots,n$ is
$$
KL({\bf p}_1,{\bf p}_2)=\sum_{i=1}^n \left\{p_{1i} \ln \left(\frac{p_{1i}}{p_{2i}}\right)+
(1-p_{1i}) \ln \left(\frac{1-p_{1i}}{1-p_{2i}}\right)\right\}
$$

Assume that there exists a constant $C_0 < \infty$ such that $\max_{1 \leq i \leq n} |\bbeta^t \bx_i| \leq C_0$ or, equivalently, $\delta \leq p_i \leq 1-\delta,\;i=1,\ldots,n$
for some positive $\delta < 1/2$ (see above). 
Assumption (A) is, therefore, satisfied 
with $\cu=1/4$ and $\cl=\delta(1-\delta)$. 

Consider the $k\ln\frac{p}{k}$-type complexity penalty (\ref{eq:varpen})  
$Pen(k)=C k\ln \frac{pe}{k}$ for $k=1,\ldots,r-1$ and $Pen(r)=C r$ in (\ref{eq:binpen}), where
$C>\frac{4}{\delta(1-\delta)}$.
From our general results from the previous sections
it then follows that
$$ EKL({\bf p},\widehat{\bf p}_{\widehat M})=O\left(\min\left(p_0 \ln\frac{pe}{p_0},r\right)\right),
$$
where $p_0=||\bbeta||_0$ is the size of the true (unknown) underlying logistic regression model. 
For weakly collinear design,
as $r$ increases, it is the minimax rate of convergence.

Similarly, the RIC-type penalty $Pen(k)=C  k \ln p,\;k=1,\ldots,r$  with $C>\frac{4}{\delta(1-\delta)}$ in (\ref{eq:binpen}) yields the sub-optimal rate
$O\left(p_0 \ln p\right)$.

\section{Possible extensions} \label{sec:extensions}
In this section we discuss some possible extensions of the results obtained in
Section \ref{sec:main}.

\subsection{Model selection in GLM under structural constraints} \label{subsec:struct}
So far we considered the complete variable selection, where the set
of admissible models $\M$ contains all $2^p$ possible subsets of
predictors $x_1,\ldots,x_p$. However, in various GLM setups there may be
additional structural constraints on the set of admissible
models. Thus, for the ordered variable selection, where the
predictors have some natural order, $x_j$ can enter the model only
after $x_1,\ldots,x_{j-1}$ (e.g., polynomial regression). Models
with interactions that cannot be selected without the
corresponding main effects is an example of hierarchical
constraints. Factor predictors associated with groups of indicator
(dummy) variables, where either none or all of the group is
selected, is an example of group structural constraints.

Model selection under structural constraints for Gaussian regression was
considered in Abramovich \& Grinshtein (2013). Its extension to GLM may be
described as follows. Let $m(p_0)$ be the number of all admissible models of
size $p_0$.
As before we can consider only $1 \leq p_0 \leq r$, where $m(r)=1$ if there
are admissible models of size $r$.
Obviously, $0 \leq m(p_0) \leq {p \choose p_0}$, where the two extreme cases $m(p_0)=1$ and $m(p_0)={p \choose
p_0}$ for all $p_0=1,\ldots,r-1$, correspond respectively to the
ordered and complete variable selection.

Let $\M$ be the set of all admissible models. We slightly change the original
definition of $\cb_M$ in (\ref{eq:mle}) by the additional requirement that
$\beta_j = 0$ iff $j \notin M$ to have $||\bbeta||_0=|M|$
for all $\bbeta \in \cb_M$. The model $\widehat{M}$ is
selected w.r.t. (\ref{eq:penlik}) from all models in $\M$ and
the penalty $Pen(k)$ is relevant only for $k$ with $m(k) \geq 1$.
From the proof (see the Appendix)
it follows that Theorem \ref{th:generalupper} can be
immediately extended to a restricted set of models $\M$
with an obviously modified condition (\ref{eq:weights}) on the weights $L_k$.
Namely, let
\be \label{eq:weightsconst}
\sum_{k=1}^{r-1}m(k)e^{-kL_k}+e^{-rL_k} \leq S
\ee
and 
$$
Pen(k) \geq 2~\frac{\cu}{\cl}~ k (A+2\sqrt{2L_k}+4L_k),\;\;\;k=1,\ldots,r;\;m(k) \geq 1
$$
for some $A>1$.
Then, under Assumption (A)
\be \label{eq:generalupperconst}
\mE KL(\btheta, X\widehat{\bbeta}_{\widehat M})
\leq \frac{4}{3}~ \inf_{M \in \M} \left\{\inf_{\widetilde{\bbeta} \in \cb_M} KL(\btheta,X\widetilde{\bbeta})+Pen(|M|)\right\}+
\frac{16}{3}~\frac{\cu}{\cl}~ \frac{2A-1}{A-1}~ S,
\ee
See Birg\'e \& Massart (2001, 2007), Abramovich \& Grinshtein (2013) for
similar results for Gaussian regression under structural constraints.

In particular, (\ref{eq:weightsconst}) holds for
$L_k= c \frac{1}{k} \max(\ln m(k),k),\;k=1,\ldots,r;\;m(k) \geq 1$
for some $c>1$ leading to the penalty of the form
\be \label{eq:varpenconst}
Pen(k) \sim \frac{\cu}{\cl} \max(\ln m(k),k)
\ee
for all $1 \leq k \leq r$ such that $m(k) \geq 1$.
For the complete variable selection, the penalty (\ref{eq:varpenconst}) is
evidently
the $k\ln\frac{p}{k}$-type penalty (\ref{eq:varpen}) from Section
$\ref{sec:main}$, while for the ordered variable selection it implies the
AIC-type penalty of the form $Pen(k) = C \frac{\cu}{\cl}~k$ for some $C>0$.

Consider now all admissible models of size $p_0$ and the corresponding
set of regression coefficients
$\cb(p_0)=\bigcup_{M \in \M: |M|=p_0}\cb_M$. Repeating the arguments from
Section \ref{subsec:sparserisk}, for the complexity penalty (\ref{eq:varpenconst}), under Assumption (A),
the general upper bound (\ref{eq:generalupperconst}) yields
\be \label{eq:upper0const}
\sup_{\bbeta \in \cb(p_0)} \mE KL(X\bbeta,X\widehat{\bbeta}_{\widehat M})
=O\left(Pen(p_0)\right)=O\left(\max(\ln m(p_0),p_0)\right)
\ee
with a constant depending on the ratio $\cu/\cl$.

The upper bound (\ref{eq:upper0const}) can be improved further if there exist
admissible models of size $r$. In this case $m(r)=1$ and similar to 
(\ref{eq:satmodel}) for complete variable selection, we have
$$
\sup_{\bbeta \in \cb(p_0)} \mE KL(X\bbeta,X\widehat{\bbeta}_{\widehat M})
=O(Pen(r))=O(r)
$$
that combining with (\ref{eq:upper0const}) yields
\be \label{eq:upper0const1}
\sup_{\bbeta \in \cb(p_0)} \mE KL(X\bbeta,X\widehat{\bbeta}_{\widehat M})=
O\left(\min\left(\max(\ln m(p_0),p_0),r\right)\right)
\ee

In the supplementary material we show that if $m(p_0) \geq 1$, under
Assumption (A) and correspondingly modified other assumptions of
Theorem 2, the minimax lower bound over $\cb(p_0)$ is \be
\label{eq:lowerconst} \inf_{\widetilde{\btheta}} \sup_{\bbeta \in
\cb(p_0)} \mE KL(X\bbeta,\widetilde{\theta}) \geq \left\{
\begin{array}{ll}
C_2~\frac{\cl}{\cu}~
\max\left\{\tau[2p_0]\frac{\ln m(p_0) }{\ln p_0},\tau[p_0]p_0\right\},&
\; 1 \leq p_0 \leq r/2  \\
C_2~\frac{\cl}{\cu}~ \tau[p_0]\; r,  &\; r/2 \leq p_0 \leq r
\end{array}
\right.
\ee
for some $C_2>0$.

Thus, comparing the upper bounds (\ref{eq:upper0const})--(\ref{eq:upper0const1})
with the lower bound (\ref{eq:lowerconst}) one realizes that
for weakly collinear design the proposed penalized maximum likelihood estimator
with the complexity penalty of type (\ref{eq:varpenconst})
is asymptotically (as $r$ increases) at least nearly-minimax (up to
a possible $\ln p_0$-factor) simultaneously for all $1 \leq p_0 \leq
r/2$ and for all $1 \leq p_0 \leq r$ if, in addition, $m(r)=1$
(i.e., there exist admissible models of size $r$).
In particular, for the ordered variable selection, both bounds
are of the same order $O(p_0)$. In Section \ref{sec:main} we showed that it also
achieves the exact minimax rate for complete variable selection. So
far we can only conjecture that the $\ln p_0$-factor can be removed
in (\ref{eq:lowerconst}) for a general case as well.
See also Abramovich \& Grinshtein (2013) for similar results for Gaussian
regression.

\subsection{Aggregation in GLM} \label{subsec:aggregation}
An interesting statistical problem related to model selection is aggregation.
Originated by Nemirovski (2000), it has been intensively studied in the
literature during the last decade. See, for example, Tsybakov (2003),
Young (2004),  Leung \& Barron (2006), Bunea {\em et al.} (2007) and
Rigollet \& Tsybakov (2011) for aggregation in Gaussian
regression. Aggregation in GLM  was
considered in Rigollet (2012) and can be described as follows.

We observe $(\bx_i,Y_i),\;i=1,\ldots,n$, where the distribution
$f_{\theta_i}(\cdot)$ of $Y_i$ belongs to the exponential family
with a natural parameter $\theta_i$ (\ref{eq:model}). Unlike
GLM regression with the canonical link, where we assume that
$\theta_i=\bbeta^t \bx_i$, in aggregation setup we do not rely on
such modeling assumption but simply seek the best linear approximation
$\btheta_{\bbeta}=\sum_{j=1}^p \beta_j \bx_j$ of $\btheta$ w.r.t.
Kullback-Leibler divergence, where $\bbeta \in \cb \subseteq
\mathbb{R}^p$, by solving the following optimization problem:
\be
\label{eq:aggreg} \inf_{\bbeta \in \cb}KL(\btheta,\btheta_{\bbeta})
\ee
Depending on the specific choice of $\cb \subseteq \mathbb{R}^p$
there are different aggregation
strategies. Following the terminology of Bunea {\em et al.} (2007)
there are {\em linear} aggregation
($\cb=\cb_L=\mathbb{R}^p$), {\em convex} aggregation
($\cb=\cb_C=\{\bbeta \in \mathbb{R}^p: \bbeta_j \geq 0,\;
\sum_{j=1}^p \beta_j=1\}$), {\em model selection} aggregation
($\cb=\cb_{MS}$ is a subset of vectors with a single nonzero entry),
and {\em subset selection} or {\em
$p_0$-sparse} aggregation ($\cb=\cb_{SS}(p_0)=\{\bbeta \in
\mathbb{R}^p: ||\bbeta||_0 \leq p_0\}$ for a given $1 \leq p_0 \leq
r$). In fact, linear and model selection aggregation can be viewed
as two  extreme cases of subset selection aggregation, where
$\cb_L=\cb_{SS}(r)$ and $\cb_{MS}=\cb_{SS}(1)$.

Since in practice $\btheta$ is unknown, the solution of (\ref{eq:aggreg}) is
unavailable. The goal then is to construct an estimator (linear aggregator)
$\btheta_{\widehat \bbeta}$ that mimics the ideal (oracle)
solution $\btheta_{\bbeta}$ of (\ref{eq:aggreg})
as close as possible. More precisely, we would like to find
$\btheta_{\widehat \bbeta}$ such that
\be \label{eq:excess}
\mE KL(\btheta,\btheta_{\widehat \bbeta}) \leq C
\inf_{\bbeta \in \cb}KL(\btheta,\btheta_{\bbeta}) + \Delta_{\cb}(\btheta,\btheta_{\widehat \bbeta}),\;\;\;
C \geq 1
\ee
with the minimal possible $\Delta_{\cb}(\btheta,\btheta_{\widehat \bbeta})$ (called {\em excess-KL}) and $C$ close to one.

For weakly collinear design,
Rigollet (2012, Theorem 4.1) established the minimal possible asymptotic rates for
$\Delta_{\cb}(\btheta,\btheta_{\widehat \bbeta})$ for linear, convex and model selection aggregation under
Assumption (A) and assumptions similar to those of Theorem \ref{th:lower}:
\be \label{eq:agrates}
\inf_{\btheta_{\widehat \bbeta}}\sup_{\btheta} \Delta_{\cb}(\btheta,\btheta_{\widehat \bbeta})=
\left\{
\begin{array}{ll}
O(r) & \cb=\cb_L\;\; {\rm(linear\;aggregation)})\\
O\left(\min(\sqrt{n \ln p},r)\right) & \cb=\cb_C\;\; {\rm(convex\;aggregation}) \\
O\left(\min(\ln p,r)\right) & \cb=\cb_{MS}\;\; {\rm(model\;selection\;aggregation)}
\end{array}
\right.
\end{equation}
He also proposed an estimator $\btheta_{\widehat \bbeta}$ that achieves these
optimal aggregation rates even with $C=1$ in (\ref{eq:excess}).

Using the results of Section \ref{sec:main} we can complete the case of
subset selection aggregation in GLM, where under the assumptions of Theorem 4.1 of
Rigollet (2012),
$\cb_{SS}(p_0)$ is essentially $\cb(p_0)$ considered in the context of
GLM model selection in previous sections.
Indeed, repeating the arguments in the proof of Theorem \ref{th:lower} (see
Appendix) implies
that for $\cb(p_0)$ there exists $C_2>0$ such that
\be \label{eq:SSaggr}
\inf_{\btheta_{\widehat \bbeta}}\sup_{\btheta} \Delta_{\cb(p_0)}(\btheta,\btheta_{\widehat \bbeta})
\geq C_2~ \frac{\cl}{\cu}~ \min\left(p_0 \ln\left(\frac{pe}{p_0}\right),r\right)
\ee
In particular, (\ref{eq:SSaggr}) also yields the lower bounds (\ref{eq:agrates})
for excess-KL for linear ($p_0=r$) and model selection ($p_0=1$) aggregation.
Furthermore, similar to model selection in GLM
within $\cb(p_0)$ considered in Section \ref{subsec:sparserisk},
from Theorem \ref{th:generalupper} it follows that for weakly collinear
design, the penalized
maximum likelihood estimator $\btheta_{\widehat \bbeta_{\widehat M}}$ with the
complexity penalty (\ref{eq:varpen}) achieves the
optimal rate (\ref{eq:SSaggr}) for subset selection aggregation over $\cb(p_0)$
for all $1 \leq p_0 \leq r$
(and, therefore, for linear and model selection aggregation in particular)
though with some $C>4/3$ in (\ref{eq:excess}). Similar to the results of
Rigollet \& Tsybakov (2011) for Gaussian regression,
we may conjecture that to get $C=1$ one should average
estimators from all models with properly chosen weights rather than select a single
one as in model selection. 

\subsection*{Acknowledgement} The work was supported by the Israel Science Foundation
(ISF), grant ISF-820/13. We are grateful to Alexander Goldenshluger, Ya'acov Ritov and Ron Peled for valuable remarks.

\section*{Appendix}

We first prove the following lemma establishing the equivalence of
the Kullback-Leibler divergence $KL(\btheta_1,\btheta_2)$ and the
squared quadratic norm $||\btheta_1-\btheta_2||^2$ under Assumption
(A) that will be used further in the proofs:
\begin{lemma} \label{lem:equiv}
Let Assumption (A) hold. Then, for any $\btheta_1, \btheta_2 \in
\mathbb{R}^n$ such that $\theta_{1i}, \theta_{2i} \in \Theta,\;i=1,\ldots,n$,
$$
\frac{\cl}{2a}||\btheta_1-\btheta_2||^2 \leq KL(\btheta_1,\btheta_2)
\leq \frac{\cu}{2a}||\btheta_1-\btheta_2||^2
$$
\end{lemma}
\begin{proof}
Recall that for a GLM
\be \label{eq:a7}
KL(\btheta_1,\btheta_2)=
\frac{1}{a}\sum_{i=1}^n\left\{b'(\theta_{1i})(\theta_{1i}-\theta_{2i})-
b(\theta_{1i})+b(\theta_{2i})\right\}
\ee
A Taylor expansion of $b(\theta_{2i})$ around $\theta_{1i}$
yields $b(\theta_{2i})=b(\theta_{1i})+b'(\theta_{1i})(\theta_{2i}-\theta_{1i})
+\frac{b''(c_i)}{2}(\theta_{2i}-\theta_{1i})^2$, where $c_i$ lies between
$\theta_{1i}$ and $\theta_{2i}$, and substituting into (\ref{eq:a7}) we have
$$
KL(\btheta_1,\btheta_2)=\frac{1}{2a}\sum_{i=1}^n b''(c_i) (\theta_{2i}-
\theta_{1i})^2
$$
Due to Assumption (A), $\Theta$ is an
interval and, therefore, $c_i \in \Theta$. Hence, $\cl \leq b''(c_i) \leq \cu$
that completes the proof.
\end{proof}

\subsection*{Proof of Theorem \ref{th:generalupper}}
We introduce first some notation.
For a given model $M$, define
$$
\bbeta_M=\arg \inf_{\widetilde{\bbeta} \in \cb_M} KL(\btheta,X\widetilde{\bbeta}),
$$
where $\cb_M$ is given in (\ref{eq:mle}), and let $\btheta_M=X\bbeta_M$.
As we have mentioned in Section \ref{subsec:generalupper}, $\btheta_M$ can be
interpreted as the closest vector to $\btheta$ within the span generated by a
subset of columns of $X$ corresponding to $M$ w.r.t. a Kullback-Leibler
divergence. 
Recall also that $\widehat{\btheta}_M=X\widehat{\bbeta}_M$ is the MLE of $\btheta$ for
the model $M$ and, in particular, $\widehat{\btheta}_{\widehat M}=X\widehat{\bbeta}_{\widehat M}$.
Finally, for any random variable $\eta$ let $\varphi_\eta(\cdot)$ be its
moment generating function.

For the clarity of exposition, we split the proof into several steps.

\medskip
\noindent{\em Step 1.}
Since $\widehat{M}$ is the minimizer defined in (\ref{eq:penlik}), for any
given model $M$
\be \label{eq:a0}
-\ell(\widehat{\bbeta}_{\widehat M})+Pen(|\widehat{M}|) \leq
-\ell(\bbeta_M)+Pen(|M|)
\ee
By a straightforward calculus, one can easily verify that
\be \label{eq:a01}
KL(\btheta,\widehat{\btheta}_{\widehat M})-KL(\btheta,\btheta_M)=
\ell(\bbeta_M)-\ell(\widehat{\bbeta}_{\widehat M})+
\frac{1}{a}(\bY-b'(\btheta))^t(\widehat{\btheta}_{\widehat M}-\btheta_M)
\ee
and, hence, (\ref{eq:a0}) yields
\be \label{eq:a1}
KL(\btheta,\widehat{\btheta}_{\widehat M})+Pen(|\widehat{M}|) \leq
KL(\btheta,\btheta_M)+Pen(|M|)+
\frac{1}{a}(\bY-b'(\btheta))^t(\widehat{\btheta}_{\widehat M}-\btheta_M)
\ee
Note that $\mE Y=b'(\btheta)$, $\mE \left\{(\bY-b'(\btheta))^t \bzeta\right\}=0$ for
any deterministic vector $\bzeta \in \mathbb{R}^n$ and, therefore,
$$
\mE\left((\bY-b'(\btheta))^t(\widehat{\btheta}_{\widehat M}-\btheta_M)\right)=
\mE\left((\bY-b'(\btheta))^t(\widehat{\btheta}_{\widehat M}-\btheta)\right)
$$
Furthermore, by the definition of $\btheta_{\widehat M}$,
$KL(\btheta,\widehat{\btheta}_{\widehat M}) \geq
KL(\btheta,\btheta_{\widehat M})$, and since (\ref{eq:a1}) holds for any model
$M$ in the RHS,
we have
\be \label{eq:a2}
\begin{split}
\frac{3}{4}~\mE KL(\btheta,\widehat{\btheta}_{\widehat M}) & \leq
\inf_M \left\{KL(\btheta,\btheta_M)+Pen(|M|)\right\} \\
& +
\mE\left(\frac{1}{a}(\bY-b'(\btheta))^t(\widehat{\btheta}_{\widehat M}-\btheta)-
Pen(|\widehat{M}|)-\frac{1}{4} KL(\btheta,\btheta_{\widehat M})\right)
\end{split}
\ee

\medskip
\noindent
{\em Step 2.} Consider now the term
$\frac{1}{a}(\bY-b'(\btheta))^t(\widehat{\btheta}_{\widehat M}-\btheta)$ in the RHS of (\ref{eq:a2}).
The selected model $\widehat{M}$ in (\ref{eq:penlik}) can, in principle, be
any model $M$ and we want, therefore, to control it uniformly over $M$.
For any $M$ we have
\be \label{eq:a3}
\frac{1}{a}(\bY-b'(\btheta))^t(\widehat{\btheta}_M-\btheta)=
\frac{1}{a}(\bY-b'(\btheta))^t(\widehat{\btheta}_M-\btheta_M)+
\frac{1}{a}(\bY-b'(\btheta))^t(\btheta_M-\btheta) 
\ee

Let $\Xi_M$ be any orthonormal basis of the span of columns of $X$ 
corresponding to the model $M$
and $\bxi_M=\Xi_M \Xi_M^t (\bY-b'(\btheta))$ be the projection of $\bY-b'(\btheta)$ on this span.

Then, by the Cauchy-Schwarz inequality
\be
\label{eq:a31}
(\bY-b'(\btheta))^t(\widehat{\btheta}_M-\btheta_M)=\bxi_M^t(\widehat{\btheta}_M-\btheta_M)
\leq ||\bxi_M|| \cdot ||\widehat{\btheta}_M-\btheta_M||
\ee
Since $\widehat{\btheta}_M$ is the MLE for a given $M$,
$\ell(\widehat{\btheta}_M) \geq \ell(\btheta_M)$ and, therefore, (\ref{eq:a01})
implies
\be \label{eq:a32}
KL(\btheta,\widehat{\btheta}_M) \leq KL(\btheta,\btheta_M)+
\frac{1}{a}(\bY-b'(\btheta))^t(\widehat{\btheta}_M-\btheta_M)
\ee

Similar to the proof of Lemma 6.3 of
Rigollet (2012), using a Taylor expansion it
follows that under Assumption (A),
$KL(\btheta,\widehat{\btheta}_M)-KL(\btheta,\btheta_M) \geq
\frac{\cl}{2a}||\widehat{\btheta}_M-\btheta_M||^2$ that together with
(\ref{eq:a31}) and (\ref{eq:a32}) yields
\be \label{eq:a321}
\frac{1}{a}(\bY-b'(\btheta))^t(\widehat{\btheta}_M-\btheta_M) \leq
\frac{2}{a\cl} ||\bxi_M||^2
\ee

Define
$$
R(M)=\frac{2}{a\cl}||\bxi_M||^2+
\frac{1}{a}(\bY-b'(\btheta))^t(\btheta_M-\btheta)-Pen(|M|)-\frac{1}{4} KL(\btheta,\btheta_M)
$$
Then, from (\ref{eq:a2}),
\be \label{eq:a37}
\mE KL(\btheta,\widehat{\btheta}_{\widehat M}) \leq
\frac{4}{3}~ \inf_M \left\{KL(\btheta,\btheta_M)+Pen(|M|)\right\}+\frac{4}{3}~\mE R(\widehat{M})
\ee
and to complete the proof
we need to find an upper bound for $\mE R(\widehat{M})$.

\medskip
\noindent
{\em Step 3.}
Consider $\varphi_{||\bxi_M||^2}(\cdot)$. 
By (6.3) of Rigollet (2012), 
$$
\mE e^{{\bf w}^t(\bY-b'(\btheta))} \leq e^{\frac{\cu a ||{\bf w}||^2}{2}}
$$
for any ${\bf w} \in \mathbb{R}^n$. The projection matrix $\Xi_M \Xi_M^t$ is idempotent and $tr(\Xi_M \Xi_M^t)=|M|$. We can apply then
Remark 2.3 of Hsu, Kakade \& Zhang (2012) to have
\be \label{eq:mgfxi}
\varphi_{||\bxi_M||^2}(s) \leq \exp\left\{a \cu s |M|+\frac{a^2 \cu^2 s^2 |M|}{1-2a\cu s}\right\}
\ee
for all $0 < s < \frac{1}{2a \cu}$.

Consider now the random variable
$\eta_M=(\bY-b'(\btheta))^t(\btheta_M-\btheta)$.
Applying (6.3) in Lemma 6.1 of Rigollet (2012) yields
\be \label{eq:mgfeta}
\varphi_{\eta_M}(s) \leq \exp\left\{\frac{1}{2}s^2 \cu a ||\btheta_M-\btheta||^2\right\}
\ee

Define
$Z=\frac{2}{a\cl}(||\bxi_M||^2-a\cu|M|)+\frac{1}{a}\eta_M=R(M)+Pen(|M|)+
\frac{1}{4} KL(\btheta,\btheta_M)-2\frac{\cu}{\cl}|M|$. Unlike Gaussian regression,
$||\bxi||^2_M$ and $\eta_M$ are not independent.
However, by the Cauchy-Schwarz inequality
$$
\varphi_Z(s) \leq e^{-2\frac{\cu}{\cl}|M|s} \cdot \sqrt{\varphi_{\frac{2}{a\cl}||\bxi_M||^2}(2s)}
\cdot \sqrt{\varphi_{\frac{1}{a}\eta_M}(2s)}
$$
and from (\ref{eq:mgfxi}) and (\ref{eq:mgfeta}),
\be \label{eq:a5}
\varphi_Z(s) \leq
\exp\left\{
\frac{8\frac{\cu^2}{\cl^2} |M| s^2}{1-8\frac{\cu}{\cl}s}+\frac{\cu s^2}{a}||\btheta_M-\btheta||^2\right\} 
\ee
for all $0 < s < \frac{\cl}{8\cu}$.

Let $x=8\frac{\cu}{\cl}s\; (0 < x < 1)$ and $\rho=\frac{\cl^2 ||\btheta_M-\btheta||^2}{64a\cu}$. Then,
using the obvious inequality $\rho x^2 < \rho x$ for $0 < x < 1$,
after a straightforward calculus (\ref{eq:a5}) yields
$$
\ln \varphi_{\frac{\cl}{8\cu}Z-\rho}(x)
\leq \frac{|M|}{8} \frac{x^2}{1-x}
$$
for all $0 < x < 1$.

We can now apply Lemma 2 of Birg\'e \& Massart (2007) to get
$P(\frac{\cl}{8\cu}Z-\rho \geq \sqrt{\frac{|M|}{2}~t}+t) \leq e^{-t}$ for all $t>0$, that
is,
$$
P\left\{\frac{\cl}{8\cu}\left(R(M)+Pen(|M|)+\frac{1}{4} KL(\btheta,\btheta_M)-\frac{\cl||\btheta_M-\btheta||^2}{8a}\right)
\geq \frac{|M|}{4}+\sqrt{\frac{|M|}{2}~t}+t\right\} \leq e^{-t}
$$
Lemma \ref{lem:equiv} implies that $\frac{1}{4} KL(\btheta,\btheta_M)-\frac{\cl||\btheta_M-\btheta||^2}{8a} \geq 0$ and, therefore,
\be \label{eq:a6}
P\left\{\frac{\cl}{8\cu}\left(R(M)+Pen(|M|)\right) \geq \frac{|M|}{4}+\sqrt{\frac{|M|}{2}~t}+t\right\} \leq e^{-t}
\ee

\medskip
\noindent
{\em Step 4.} Based on (\ref{eq:a6}) we can now find an upper bound for
$\mE R(\widehat{M})$.

Let $k=|M|$ and take $t=kL_k+\omega$ for any $\omega>0$, where $L_k>0$ are
the weights from Theorem \ref{th:generalupper}. Using inequalities
$\sqrt{c_1+c_2} \leq \sqrt{c_1}+\sqrt{c_2}$ and $\sqrt{c_1 c_2} \leq \frac{1}{2}(c_1\epsilon+
\frac{c_2}{\epsilon})$ for any positive  $c_1, c_2$ and $\epsilon$, we have
$$
\sqrt{kt} \leq k \sqrt{L_k}+\sqrt{k\omega} \leq k \sqrt{L_k}+\frac{1}{2}\left(k \epsilon+\frac{\omega}{\epsilon}\right)
$$
and, therefore,
$$
P\left\{\frac{\cl}{8\cu}\left(R(M)+Pen(k)\right)
\geq \frac{k}{4}\left(1+\sqrt{2}~\epsilon+2\sqrt{2 L_k}+4L_k\right)+\omega \left(1+\frac{1}{2\sqrt{2}~\epsilon}\right)\right\} \leq
e^{-(kL_k+\omega)}
$$
For the penalty $Pen(k)$
satisfying (\ref{eq:penalty}) with some $A>1$ and
$\epsilon=(A-1)/\sqrt{2}$, we then have
\be \label{eq:a8}
P\left\{\frac{\cl}{4\cu}R(M) \geq \omega \frac{2A-1}{A-1}\right\}
\leq e^{-(kL_k+\omega)}
\ee
for all $M$.

Finally, under the condition (\ref{eq:weights}) on the weights $L_k$,
(\ref{eq:a8}) implies
$$
P\left\{R(\widehat{M}) \geq \frac{4\cu}{\cl}~\omega \frac{2A-1}{A-1}\right\}
\leq \sum_M P\left\{R(M) \geq \frac{4\cu}{\cl}~\omega\frac{2A-1}{A-1}\right\}
\leq \sum_M e^{-(kL_k+\omega)} \leq S e^{-\omega}
$$
and, hence,
$$
\mE R(\widehat{M}) \leq \int_0^\infty P(R(\widehat{M})>t)dt \leq 
\frac{4\cu}{\cl}~ \frac{2A-1}{A-1}~ S
$$
that together with (\ref{eq:a37}) completes the proof.
\qed

\subsection*{Proof of Theorem \ref{th:lower}}
Due to Lemma \ref{lem:equiv}, the minimax lower bound for the Kullback-Leibler
risk can be reduced to the lower bound for the corresponding quadratic risk:
\be \label{eq:a20}
\inf_{\widetilde \btheta}\sup_{\bbeta \in \cb(p_0)}\mE
KL(X\bbeta,\widetilde{\btheta}) \geq \frac{\cl}{2a}\inf_{\widetilde \btheta}\sup_{\bbeta \in \cb(p_0)} \mE ||X\bbeta-\widetilde{\btheta}||^2
\ee

Following a general reduction scheme for establishing the minimax risk lower bounds, the quadratic risk in (\ref{eq:a20}) is first reduced to the probability of misclassification error among 
a properly chosen finite subset $\Theta^*(p_0) \subset \{\btheta \in \mathbb{R}^n: \btheta=X\bbeta,\;
\bbeta \in \cb(p_0)\}$ such that for any $\btheta_1, \btheta_2 \in \Theta^*(p_0)$,
$||\btheta_1-\btheta_2||^2 \geq 4 s^2(p_0)$:
$$
\inf_{\widetilde \btheta}\sup_{\bbeta \in \cb(p_0)} \mE ||X\bbeta-\widetilde{\btheta}||^2 ~\geq~ \inf_{\widetilde \btheta} \max_{\btheta_j \in  \Theta^*(p_0)} \mE||\btheta_j-\widetilde{\btheta}||^2 
~\geq ~4 s^2(p_0) \inf_{\widetilde \btheta} \max_{\btheta_j \in  \Theta^*(p_0)}
P_{\btheta_j}(\widetilde{\btheta} \neq \btheta_j)
$$
and then bounding the latter from below (e.g., applying various versions of Fano' lemma). See Tsybakov (2009, Section 2) for more details.
  
In particular, the idea of our proof is to find a finite subset $\cb^*(p_0) \subseteq \cb(p_0)$
of vectors $\bbeta$ and the corresponding subset
$\Theta^*(p_0)=\{\btheta \in \mathbb{R}^n: \btheta=X\bbeta,\;
\bbeta \in \cb^*(p_0)\}$ such that for any $\btheta_1, \btheta_2 \in \Theta^*(p_0)$,
$||\btheta_1-\btheta_2||^2 \geq 4 s^2(p_0)$ and
$KL(\btheta_1,\btheta_2) \leq (1/16) \ln {\rm card}(\Theta^*(p_0))$.
It will follow then from Lemma A.1 of Bunea {\em et al.} (2007) that
$s^2(p_0)$ is the minimax lower bound for the quadratic risk over $\cb(p_0)$.

To construct such subsets we can exploit the techniques similar
to that used in the corresponding proofs for the quadratic risk in linear regression
(e.g., Abramovich \& Grinshtein, 2010; Rigollet \& Tsybakov, 2011). Consider
three cases.

\medskip
\noindent {\sc Case 1}. $p_0 \leq r/2$
\newline
Define the subset $\widetilde{\cb}(p_0)$ of all vectors
$\bbeta \in \mathbb{R}^p$ that have $p_0$ entries equal to $C_{p_0}$,
where $C_{p_0}$ will be defined below
and others are zeros: $\widetilde{\cb}(p_0)=
\{\bbeta \in \mathbb{R}^p: \bbeta \in \{0,C_{p_0}\}^p,
||\bbeta||_0=p_0\}$.
From Lemma A.3 of Rigollet \& Tsybakov (2011), there exists a subset
$\cb^*(p_0) \subset \widetilde{\cb}(p_0)$ such that
$\ln{\rm card}(\cb^*(p_0)) \geq \tilde{c} p_0 \ln\left(\frac{pe}{p_0}\right)
$ for some constant $0< \tilde{c} < 1$, and for any pair $\bbeta_1,\;\bbeta_2
\in \cb^*(p_0)$, the Hamming distance
$\rho(\bbeta_1,\bbeta_2)=\sum_{j=1}^p \mathbb{I}\{\beta_{1j} \neq \beta_{2j}\}
\geq \tilde{c}p_0$.

Take $C^2_{p_0}=\frac{1}{16}\tilde{c}\frac{a}{\cu}\phi^{-1}_{\max}[2p_0]\ln\left(
\frac{pe}{p_0}\right)$.
By the assumptions of the theorem, $\cb^*(p_0) \subset \widetilde{\cb}(p_0)  \subseteq \cb(p_0)$.
Consider the corresponding subset $\Theta^*(p_0)$. Evidently,
${\rm card}(\Theta^*(p_0))={\rm card}(\cb^*(p_0))$, and for any $\btheta_1,
\btheta_2 \in \Theta^*(p_0)$ associated with $\bbeta_1, \bbeta_2 \in
\cb^*(p_0)$ we then have
\be \label{eq:a21}
||\btheta_1-\btheta_2||^2 = ||X(\bbeta_1-\bbeta_2)||^2 \geq \phi_{min}[2p_0] \;
||\bbeta_1-\bbeta_2||^2 \geq \tilde{c}\phi_{min}[2p_0]C^2_{p_0}\;p_0=
4s^2(p_0),
\ee
where $s^2(p_0)=\frac{1}{64}\frac{a}{\cu} \tilde{c}^2 \tau[2p_0] p_0\ln\left(\frac{pe}{p_0}
\right)$.

On the other hand,
\be \label{eq:a22}
K(\btheta_1,\btheta_2) \leq \frac{\cu}{2a} ||\btheta_1-\btheta_2||^2
\leq \frac{\cu}{2a}~
\phi_{max}[2p_0] C^2_{p_0} \rho(\bbeta_1,\bbeta_2)
\leq \frac{\cu}{a}~ \phi_{max}[2p_0] C^2_{p_0} p_0 \leq \frac{1}{16}
\ln {\rm card}(\Theta^*(p_0)),
\ee
where the first inequality follows from Lemma \ref{lem:equiv}.
Lemma A.1 of Bunea {\em et al.} (2007) and (\ref{eq:a20}) complete then the
proof for this case.

\medskip
\noindent {\sc Case 2}. $r/2 \leq p_0 \leq r,\;p_0 \geq 8$
\newline
In this case consider the subset $\widetilde{\cb}(p_0)=\{\bbeta \in \mathbb{R}^p:
\bbeta \in \{\{0,C_{p_0}\}^{p_0},0,\ldots,0\}$, where
$C^2_{p_0}=\frac{\ln 2}{64}\frac{a}{\cu}\phi^{-1}_{\max}[p_0]$. From
the assumptions of the theorem  $\widetilde{\cb}(p_0) \subseteq \cb(p_0)$.
Varshamov-Gilbert bound (see, e.g., Tsybakov, 2009, Lemma 2.9) guarantees
the existence of a subset $\cb^*(p_0) \subset \widetilde{\cb}(p_0)$ such that
$\ln{\rm card}(\cb^*_{p_0}) \geq \frac{p_0}{8} \ln 2$ and the Hamming distance
$\rho(\bbeta_1,\bbeta_2) \geq \frac{p_0}{8}$ for any pair $\bbeta_1,\;\bbeta_2
\in \cb^*_{p_0}$.

Note that for any $\bbeta_1,\;\bbeta_2 \in \cb^*_{p_0}$,
$\bbeta_1-\bbeta_2$ has at most $p_0$ nonzero components and
repeating the arguments for the Case 1, one obtains the minimax
lower bound $s^2(p_0)=C\frac{a}{\cu}\tau[p_0]p_0 \geq
\frac{C}{2}\frac{a}{\cu}\tau[p_0]r$ for the quadratic risk. Applying
(\ref{eq:a20}) completes the proof.

\medskip
\noindent {\sc Case 3}. $r/2 \leq p_0 \leq r,\;2 \leq p_0 < 8$
\newline
For this case, obviously, $2 \leq r < 16$.
Consider a trivial subset $\cb^*_{p_0}$ containing
just two vectors $\bbeta_1 \equiv 0$ and $\bbeta_2$ that has first $p_0$ nonzero
entries equal to $C_{p_0}$, where $C^2_{p_0}= \frac{\ln 2}{64}\frac{a}{\cu}\phi^{-1}_{max}[p_0]$.
Under the assumptions of the theorem $\cb^*_{p_0} \subset \cb(p_0)$. For the
corresponding $\btheta_1=X\bbeta_1$ and $\btheta_2=X\bbeta_2$,
(\ref{eq:a21}) and (\ref{eq:a22}) yield
$$
KL(\btheta_1,\btheta_2) \leq \frac{\cu}{2a} \phi_{max}[p_0] 8 C_{p_0}^2 =
\frac{1}{16} \ln{\rm card}(\Theta^*_{p_0})
$$
and
$$
||\btheta_1-\btheta_2||^2 \geq \phi_{min}[p_0] p_0 C^2_{p_0} =
C \frac{a}{\cu} \tau[p_0]p_0 \geq
\frac{C}{2} \frac{a}{\cu} \tau[p_0]r
$$
and the proof follows from Lemma A.1 of Bunea {\em et al.} (2007).
\qed

\end{document}